\documentclass[12pt]{article}
\usepackage{indentfirst}
\usepackage{cite,amsfonts,amsmath,amsthm, amssymb}
\usepackage[top=2.5cm,bottom=2.5cm,left=2.5cm,right=2.5cm]{geometry}
\usepackage[margin=1cm,%
                font=small,%
                format=hang,%
                labelsep=period,%
                labelfont=bf]{caption}
\usepackage{tikz}

\usepackage{graphicx}
\usepackage{hyperref}
\usepackage{mathrsfs}
\usepackage{latexsym, euscript, epic, eepic, color}

\newtheorem{lemma}{Lemma}[section]
\newtheorem{theorem}{Theorem}[section]
\newtheorem{definition}{Definition}

\newtheorem{corollary}{Corollary}[section]

\newtheorem{remark}{Remark}

\begin{document}
\title{The degrees, number of edges, spectral radius and weakly Hamilton--connectedness of bipartite graphs
\footnotetext {* Corresponding author. This project is  supported by Natural Science Foundation of Guangdong (Nos. 2014A030313640).}}

\author{Jia Wei$^{1}$, Zhifu You$^{2,*}$\\
\small \it $^{1}$Department of Mathematics, South China University of Technology\\
\small \it Guangzhou 510640, P.R. China\\
\small \it $^{2}$School of  Mathematics and Systems Science, Guangdong Polytechnic Normal University\\
\small \it Guangzhou 510665, P.R. China\\
\small \it E-mail: 201520122177@mail.scut.edu.cn, youzhf@hotmail.com}
\date{}
\maketitle

\begin{center}
\begin{minipage}{120mm}
{\small {\bf Abstract.}
A path of a graph $G$ is called a Hamilton path if it passes through all the vertices of $G$.
A graph is Hamilton-connected if any two vertices are connected by a Hamilton path.
Note that any bipartite graph is not Hamilton-connected. We consider the weak version of Hamilton-connected property among bipartite graphs.
A weakly Hamilton-connected graph is a balanced bipartite graph $G=(X,Y,E)$ that there is a Hamilton path between any vertex $x\in X$ and $y\in Y$.
%Various sufficient conditions for a simple graph (bipartite graph) to be traceable or Hamiltonian have been given by Ore\cite{lesO1960}, Chvat\'al\cite{lesC1972}, Moon and Moser\cite{lesM1963}, Berge\cite{lesB1976} and so on. Analogously, some sufficient conditions for a simple graph to be Hamilton--connected have been given by Ore\cite{lesO1963}, Erd\H{o}s and Gallai\cite{lesE1959}, Berge\cite{lesB1976} and others.
In this paper, we present  some  degrees,  number of edges,   and spectral radius conditions for a simple balanced bipartite graph to be weakly Hamilton--connected. }

{\small{\bf Keywords:} \ Weakly Hamilton--connected; Balanced bipartite graph; Degree; Number of edges; Spectral radius }

{\bf MSC:  }  05C50 
\end{minipage}
\end{center}

\section{Introduction}
In this paper, only connected graphs without loops and multiple edges are  considered.
A graph $G$ is \emph {bipartite} if its vertex set can be partitioned into two subsets $X$ and $Y$,
such that every edge has one end in $X$ and one end in $Y$. The partition $(X,Y)$ is called the \emph {bipartition} of  graph $G$.  $X$ and $Y$ are  its \emph {parts}.
 The graph $G$ is  \emph {balanced bipartite}  when $|X|=|Y|$.
Let $G=G(X,Y,E)$ be a balanced bipartite graph with the bipartition $(X,Y)$ and edge set $E(G)$.
The \emph {quasi--complement} of $G$, denoted by $\widehat{G}$, is a graph with vertex set $V(\widehat{G})=V(G)$
and edge set $E(\widehat G)=\{xy|x\in X,y\in Y,xy\notin E(G)\}$.
Let $e(G)$ and $\delta(G)$ be the number of edges and  minimum degree of  graph $G$, respectively.
The set of neighbours of a vertex $u$ in $G$ is denoted by $N_G(u)$ and
let $d_G(u)=|N_G(u)|$. For $S\subseteq V(G)$, the \emph {induced subgraph} $G[S]$ is the graph whose vertex set is $S$ and edge set is $\{uv\in E(G)\mid u,v\in S\}$. \emph {The disjoint union} $G_1+G_2$ of two graphs $G_1$ and $G_2$, is the graph with the vertex set $V(G_1)\cup V(G_2)$ and edge set $E(G_1)\cup E(G_2)$.
 %The disjoint union of $k$ copies of a graph $G$ is denoted by $kG$. For two bipartite graphs $G_1(X_1, Y_1)$, $G_2=(X_2,Y_2)$, let $G_1\sqcup % G_2$ be the graph obtained from $G_1\cup G_2$ by adding all possible edges between $X_2$ and $Y_1$ and all possible edges between $X_1$ and % $Y_2$.
$K_{m,n}$ is the complete bipartite graph with parts of sizes $m$ and $n$.

A path or  cycle of $G$ is called a \emph {Hamilton path or Hamilton cycle} if it passes through all the vertices of $G$. A graph $G$ is called \emph{traceable or Hamilton} if $G$ has a Hamilton path or Hamilton cycle. A graph is \emph {Hamilton--connected} if any two vertices are connected by a Hamilton path.
Let $P$ be a path of $G$ with a given direction. For two vertices $x$ and $y$ on $P$,
we use $x\overrightarrow{P}y$ ($x\overleftarrow{P}y$) to denote the segment from $x$ to $y$  of $P$ along (against) the direction.

Note that any bipartite graph $G=G(X,Y,E)$ is not Hamilton--connected. If $\big||X|-|Y|\big|\geq2$, then there is no a Hamilton path for any two vertices. If  $\big||X|-|Y|\big|=1$, then we can only consider whether there is a Hamilton path between any two different vertices of  the part with  size $max\{|X|,|Y|\}$. If $|X|=|Y|$, then we can only consider whether there is a Hamilton path for any two different vertices in different parts.

 \begin{definition}\label{de1} A balanced bipartite graph is weakly Hamilton--connected if  any two different vertices in different parts can be connected by a Hamilton path.
\end{definition}
Let $Q_n^t $ $(2\leq t\leq \frac{n+1}{2})$  be the graph obtained from $K_{n,n}$ by deleting all edges
of one subgraph $K_{t-1, n-t}$.  It is obvious that  $Q_n^t $ is weakly Hamilton--connected.

Let $ A(G)$ and $D(G)$ be the adjacency matrix and degree matrix of $G$, respectively. The largest eigenvalue of $A(G)$ is called the \emph{spectral radius} of $G$, denoted by $\rho(G)$. Let $Q(G)=A(G)+D(G)$ be the signless Lapalacian matrix of $G$. The largest eigenvalue of $Q(G)$ is called the \emph{signless Lapalacian spectral radius} of $G$, denoted by $q(G)$.
%Let $L(G)=D(G)-A(G)$ be the Lapalacian matrix of $G$. The largest eigenvalue of $L(G)$ is called the \emph{Lapalacian spectral radius} of $G$, denoted by $\lambda(G)$.

The problem that a graph is Hamilton or not has attracted many interests (see \cite{lesM1963}, \cite{lesM2012}, \cite{lesN.L2017} and the references therein). In Sections 2-4, we present some degrees, number of edges, and spectral radius conditions for a simple balanced bipartite graph to be weakly Hamilton--connected, respectively.
%In this paper, we present spectral conditions  on  a simple balanced bipartite graph to be weakly Hamilton--connected.
%In Section 2, some lemmas and results are listed. The proof of main result is given in Section 3.

\section{Degrees and weakly Hamilton--connected bipartite graphs}
We first state a known consequence as our tool which has been used to prove a simple graph to be Hamilton--connected.
\begin{lemma}(Berge\cite{lesB1976})\label{le1}
Let $\mathscr H$ be a class of simple graphs of order $n$ satisfying the following conditions:\\
(1) If $G\in \mathscr H$, each edge of $G$ is contained in some Hamilton cycle.\\
(2) If $G\in \mathscr H$, the graph $G'$ obtained from $G$ by adding any new edge also belongs to $\mathscr H$.  \\
Then $G$ is Hamilton--connected.
\end{lemma}

Note  that simple balanced bipartite graphs are also simple graphs. By Lemma \ref{le1}, the following corollary holds.
\begin{corollary}\label{co1}
Let $\mathscr B$ be a class of simple balanced bipartite graphs of order $2n$ satisfying the following conditions:\\
(1) If $G\in \mathscr B$, each edge of $G$ is contained in some Hamilton cycle.\\
(2) If $G\in \mathscr B$, the graph $G'$ obtained from $G$ by adding any new edge also belongs to $\mathscr B$.  \\
Then $G$ is weakly Hamilton--connected.
\end{corollary}

Next, by Corollary 2.1, we obtain Theorem \ref{th1} and Lemma \ref{le2}.
\begin{theorem}\label{th1}
A bipartite graph $G=G(X,Y,E)$ with $X=\{x_1,x_2, \dots, x_n\}$ and $Y=\{y_1,y_2, \dots, y_n\}$ is weakly Hamilton--connected if $G$ satisfies the following property: for any nonempty subset $\Gamma=\{x_i|d(x_i)\leq k\}$ where $|\Gamma|=k-1$ and $k\leq \frac{n+1}{2}$, every vertex $y_i$ with $d(y_i)\leq n-k+1$ is adjacent to at least one vertex in $\Gamma$. And the same result holds while $x_i$ and $y_i$ are interchanged.
\end{theorem}

\begin{proof}
[\bf Proof]By contradiction. Suppose that $G$ is not weakly Hamilton--connected, then $G$ is contained in a maximal non--weakly Hamilton--connected bipartite graph $G^*$
( It means the addition of any new edge to $G^*$ makes a weakly Hamilton--connected bipartite graph).
 By Corollary \ref{co1}, there is an edge $e$ in $G^*$, which is not contained in any Hamilton cycle of $G^*$.
 Since $G^*$ is not weakly Hamilton--connected, it is  not a complete  bipartite graph. Choose two nonadjacent vertices
 $x\in X, y\in Y$ in $G^*$ with $d_{G^*}(x)+d_{G^*}(y)$ as large as possible.
Since $G^*+xy$ is weakly Hamilton--connected, there is a Hamilton cycle containing the edges $e$ and $xy$.
Then there exists a Hamilton path $P$ in $G^*$ containing $e$ with the form $xy_1x_2y_2\cdots x_ny$ $(x=x_1,y=y_n)$.

Let $I=\{i |1\leq i \leq n-1, xy_i\in E(G^*) \quad and  \quad x_iy_i\neq e\}$, then
\begin{equation}\label{eq1}
|I|\geq d_{G^*}(x)-1.
\end{equation}

For arbitrary $i\in I$, it must have $x_iy\notin E(G^*)$. Otherwise, there exists a Hamilton cycle $y_ix\overrightarrow{P} x_iy \overleftarrow{P}y_i$ containing $e$. Thus
\begin{equation}\label{eq2}
n-d_{G^*}(y)\geq |I|.
\end{equation}

By (\ref{eq1}) and (\ref{eq2}),
\begin{equation}\label{eq3}
d_{G^*}(x)+d_{G^*}(y)\leq n+1.
\end{equation}

If $d_{G^*}(x)\leq d_{G^*}(y)$, then $d_{G^*}(x)\leq \frac{n+1}{2}$. If $d_{G^*}(x)\geq d_{G^*}(y)$, then $d_{G^*}(y)\leq \frac{n+1}{2}$. Without loss of generality, assume that $d_{G^*}(x)\leq d_{G^*}(y)$. Let $k=d_{G^*}(x)$. By (\ref{eq3}), $ d_{G^*}(y)\leq n-k+1$.

Since for arbitrary $i\in I$, $x_iy\notin G^*$, and $|I|\geq d_{G^*}(x)-1=k-1$,
there are at least $k-1$ such vertices $x_i$  not adjacent to $y$.
By the choice of $x$ with $d_{G^*}(x)+d_{G^*}(y)$ as large as possible,
 we get $d_{G^*}(x_i)\leq d_{G^*}(x)=k$. Thus there are at least $k-1$ vertices with degree no more than $k$ in  $X$ not adjacent to $y$, a contradiction.

Hence, $G$ is weakly Hamilton--connected.
\end{proof}

By the proof of Theorem \ref{th1},  the following two corollaries hold:

\begin{corollary}\label{co2}
A bipartite graph $G=G(X,Y,E)$ with $|X|=|Y|=n$ is weakly Hamilton--connected if $d(x)+d(y)\geq n+2$ for any two nonadjacent vertices $x\in X$ and $y\in Y$.
\end{corollary}

\begin{corollary}\label{co3}
Let $G=G(X,Y,E)$ be a bipartite graph of order $2n$. The graph $G$ satisfies that the number of vertices $x$ in $X$ such that $d(x)\leq k$ is less then $k-1$, where $1\leq k\leq \frac{n+1}{2}$, and similar conditions holds for  replacing  $x$ by $y$, then $G$ is weakly Hamilton--connected.
\end{corollary}

Similar to the  proof of  Theorem \ref{th1}, we have  the following lemma.

\begin{lemma}\label{le2}
Let $G=G(X,Y,E)$ be a bipartite graph with $X=\{x_1, \ldots, x_n \}$,  $Y=\{y_1, \ldots, y_n \}$, and $|X|=|Y|=n$. If $d(x)+d(y)\geq n+2$ for any two nonadjacent vertices $x\in X$ and $y\in Y$, then $G$ is weakly Hamilton--connected if and only if $G+xy$ is weakly Hamilton-connected.
\end{lemma}

\begin{proof}
[\bf Proof] Sufficiency is obvious.

\qquad Necessity is completed by contradiction. Suppose that $G$ is not weakly Hamilton--connected. Adding some suitable edges to $G$, we can get a maximal non--weakly Hamilton--connected bipartite graph $G^*$. By Corollary \ref{co1}, there is an edge $e$ of $G^*$  not contained in any Hamilton cycle. Since $G+xy$ is weakly Hamilton--connected,  $G^*+xy$ is weakly Hamilton--connected. For $G^*$, there is a Hamilton path $P$ in $G^*$ containing $e$ with the form $xy_1x_2y_2\cdots x_ny (x=x_1, y=y_n)$.

Let $I=\{i |xy_i\in E(G^*) \quad and \quad x_iy_i\neq e\}$, then
\begin{equation}\label{equ1}
|I|\geq d_{G^*}(x)-1.
\end{equation}
For arbitrary $i\in I$,  $x_iy\notin G^*$ holds. Otherwise, there exists a Hamilton cycle $y_ix\overrightarrow{P} x_iy \overleftarrow{P}y_i$ containing $e$.
Thus
\begin{equation}\label{equ2}
n-d_{G^*}(y)\geq |I|.
\end{equation}
By (\ref{equ1}) and (\ref{equ2}),
\begin{equation}\label{equ3}
d_{G^*}(x)+d_{G^*}(y)\leq n+1,
\end{equation}
which leads to a contradiction.

Hence, $G$ is weakly Hamilton--connected.
\end{proof}

\begin{definition}\label{de2}

The $B_{n+2}$--closure of a balanced bipartite graph $G$ of order $2n$, denoted by $cl_{B_{n+2}}(G)$,
is obtained from $G$ by recursively joining pairs of nonadjacent vertices in different parts,
whose degree sum is at least $n+2$ until no such pair remains.
\end{definition}
By  Lemma \ref{le2}, each time an edge such that  vertices in different parts and degree sum  at least $n+2$ is added to $G$ in the formation of
$cl_{B_{n+2}}(G)$. Then we have  the following result.

\begin{theorem}\label{th3}
A balanced bipartite graph $G=G(X,Y,E)$ of order $2n$ is weakly Hamilton--connected if and only if $cl_{B_{n+2}}(G)$ is weakly Hamilton-connected.
\end{theorem}

Because  a complete bipartite graph is weakly Hamilton--connected and by Theorem \ref{th3}, we get the following corollary.
\begin{corollary}\label{co4}
Let $G=G(X,Y,E)$ be a balanced bipartite graph with $|X|=|Y|=n$. If the $B_{n+2}$--closure of $G$ is a complete bipartite graph, then $G$ is weakly Hamilton--connected.
\end{corollary}

Next, we give another sufficient degree condition to prove a balanced bipartite graph is weakly Hamilton--connected.
And Theorem \ref{th4} is used to prove Theorem \ref{th5} in Section 3.

\begin{theorem}\label{th4}
Let $G=G(X,Y,E)$ be a balanced bipartite graph of order $2n$ and degree sequence $(d_1,d_2,\dots,d_{2n})$. Suppose that there is no integer $2\leq k\leq \frac{n+1}{2}$ such that $d_{k-1}\leq k$ and $d_{n-1}\leq n-k+1$, then $G$ is weakly Hamilton--connected.
\end{theorem}

\begin{proof}
[\bf Proof]Let $G'$ be the $B_{n+2}$--closure of $G$, i.e.,  $G'=cl_{B_{n+2}}(G)$.\\
{\bf Claim:} $G'$ is a  complete bipartite graph.

By contradiction. If $G'$ is not complete, then we choose two nonadjacent vertices $x\in X, y\in Y$ with $d_{G'}(x)+d_{G'}(y)$ as large as possible.
By the definition of $G'$, $d_{G'}(x)+d_{G'}(y)\leq n+1$.

There are two cases:

Case 1:  $d_{G'}(x)\leq d_{G'}(y)$, let $k=d_{G'}(x)$.

Denote
\begin{equation*}
X_1=\{ x_i|x_iy\notin E(G') \}, \qquad Y_1=\{ y_i|xy_i\notin E(G') \},
\end{equation*}
then
\begin{equation*}
|X_1|=n-d_{G'}(y),\qquad |Y_1|=n-d_{G'}(x).
\end{equation*}
Thus
\begin{equation*}
|Y_1|=n-k,\qquad |X_1|\geq n-(n+1-d_{G'}(x))=k-1.
\end{equation*}
That is, there are $n-k$ vertices in $Y$ not adjacent to $x$  and at least $k-1$ vertices in $X$ not adjacent to $y$.
By the choice of $x$, for any $x_i\in X_1$, $d_{G'}(x_i)\leq d_{G'}(x)$.
 It implies that there are at least $k-1$ vertices with degree no more than $k$ in $G'$.

Since $d_{G'}(x)+d_{G'}(y)\leq n+1$ and $k=d_{G'}(x)$, $d_{G'}(y)\leq n-k+1$.
By the choice of $y$, for any $y_i\in Y_1$, $d_{G'}(y_i)\leq d_{G'}(y)$.
It implies there are exactly $n-k$ vertices with degree less than or equal to $n-k+1$ in $Y_1$.
Since $d_{G'}(x)\leq d_{G'}(y)$, there are at least $n-1$ vertices with degree less than or equal to $n-k+1$ in $G'$.

Note that $G$ is a spanning subgraph of $G'$.
For any one vertex $v\in V(G)=V(G')$, $d_{G}(v)\leq d_{G'}(v)$ holds. Thus
$$d_{k-1}\leq k \quad and \quad d_{n-1}\leq n-k+1,$$
a contradiction.

Case 2: $d_{G'}(x)\geq d_{G'}(y)$, let $k=d_{G'}(y)$.

 Similarly, we have $d_{k-1}\leq k$ and $d_{n-1}\leq n-k+1$, a contradiction.

Thus $G'$ is a complete bipartite graph.

By Corollary \ref{co4}, $G$ is weakly Hamilton--connected.

The proof is finished.
\end{proof}

\section{Number of edges and weakly Hamilton--connected bipartite graphs}

%The following result is proved by similar ways used by P\'osa \cite{lesP1962}, Erd\"{o}s \cite{lesE1962} and Moon and Moser \cite{lesM1963}.
In this section, we obtain  the weakly Hamilton-connected property of  bipartite graphs by the number  of edges, namely, Theorems \ref{th5} and \ref{le4}.

\begin{theorem}\label{th5}
Let $G=G(X,Y,E)$ be a bipartite graph with $X=\{x_1,x_2, \dots, x_n\}$ and $Y=\{y_1,y_2, \dots, y_n\}$. The degree sequence of $G$ is $(d_1,\dots,d_{2n})$ with $d_{2n}\geq \cdots \geq d_1 \geq k$ where $2\leq k \leq \frac{n+1}{2}$. If
\begin{equation}\label{equ7}
e(G)>\max \left\{n(n-t+1)+t(t+1)\bigg|k\leq t \leq \frac{n+1}{2} \right\},
\end{equation}
then $G$ is weakly Hamilton--connected.
\end{theorem}

\begin{proof}
[\bf Proof] By contradiction.  Suppose that $G$ is not weakly Hamilton--connected. By Theorem \ref{th4}, there is an integer $t$, $k\leq t\leq \frac{n+1}{2}$ such that $d_{t-1}\leq t$. Then there are at least $t-1$ vertices $v_1,\dots,v_{t-1}$ with degree not exceeding $t$. For any $1\leq i\leq t-1$, the number of edges of $G$ which are not adjacent to $v_i$ is at most $n(n-t+1)$. And the number of edges of $G$ incident to  these $t-1$ vertices is at most $t(t-1)$. Thus  $e(G)\leq n(n-t+1)+t(t+1)$  where $k\leq t \leq \frac{n+1}{2}$, a contradiction.

Hence $G$ is weakly Hamilton--connected.
\end{proof}

\begin{remark}
Bound (\ref{equ7}) is not the best. Let $G=Q_n^t$, then $e(G)=n(n-t+1)+t(t-1)$ and $G$ is weakly Hamilton--connected.
% In $G$, let $d(x_1)=\cdots =d(x_{t-1})=t$, $d(x_t)=\cdots =d(x_{n})=d(y_1)=\cdots =d(y_{t})=n$, $d(y_{t+1})=\cdots =d(y_{n})=n-t+1$.
\end{remark}

Next, we present one lemma which is a main tool to prove Theorem \ref{le4}.

\begin{lemma}\label{le3}
Let $G=(X, Y, E)$ be a balanced bipartite graph of order $2n$ and $G=cl_{B_{n+2}}(G)$. If $n\geq 2k$ and $e(G)>n(n-k)+k(k+1)$ for $k\geq 1$, then $G$ contains a complete bipartite graph of order $2n-k+1$. Furthermore, if $\delta(G)\geq k$, then $K_{n,n-k+1}\subseteq G$.
\end{lemma}

\begin{proof}
[\bf Proof]Let $X_1=\{x\mid x\in X\text{ and }d(x)\geq \frac{n+2}{2}\}$ and $Y_1=\{y\mid y\in Y\text{ and }d(y)\geq \frac{n+2}{2}\}$.
For any  $x\in X_1$ and $y\in Y_1$,  then $d(x)+d(y)\geq n+2$. Note that $G=cl_{B_{n+2}}(G)$,  then $xy\in E(G)$.
For any $x\in X_1$ and $y\in Y_1$, since $d(x)\geq \frac{n+2}{2}\geq k+1$ and $d(y)\geq \frac{n+2}{2}\geq k+1 $ hold,  $K_{k+1,k+1}\subseteq G$.\\
Let $t$ be the maximal integer such that $K_{t,t}\subseteq G$, then $t\geq k+1$.\\
{\bf Claim 1:} $t\geq n-k+1$.

By contradiction. If $k+1\leq t\leq n-k$, then let $X_2\subseteq X$ and  $Y_2\subseteq Y$ such that $ G(X_2,Y_2,E_2)=K_{t,t}$.
There are two cases: \\
(1)  For  any $x\in X\backslash{X_2}$, if there exists $y\in Y_2$ such that $xy\notin E(G)$, by the definition of $cl_{B_{n+2}}(G)$, then $d(x)+d(y)\leq n+1$. Note that $d(y)\geq t$,  then we have $d(x)\leq n-t+1$.

\begin{align*}
e(G)&=e(X_2,Y_2)+e(X_2,Y\backslash Y_2)+e(X\backslash X_2,Y)\\
    &\leq t^2+t(n-t)+(n-t)(n-t+1)\\
    &=n^2+n+t^2-(n+1)t\\
    &\leq n^2+n+(n-k)^2-(n+1)(n-k)\\
    &=n^2-nk+k(k+1)\\
    &<e(G),
\end{align*}
a contradiction. \\
(2)  If there exists $x\in X\backslash{X_2}$ such that for any $y\in Y_2$, $xy\in E(G)$,
then $d(w)\leq n-t+1$ for any $w\in Y\backslash Y_2$.
Otherwise, if there is a vertex $w\in Y\backslash Y_2$ with $d(w)\geq n-t+2$, then $w$
is adjacent to every vertex in $X_2$, which violates the choice of $t$.
Then
\begin{align*}
e(G)&=e(Y_2, X_2)+e(Y_2,X\backslash X_2 )+e( Y\backslash Y_2,X)\\
    &\leq t^2+t(n-t)+(n-t)(n-t+1)\\
    &<e(G),
\end{align*}
a contradiction.

Thus $t\geq n-k+1$.

Now, let $s$ be the largest integer such that $K_{s,t}\subseteq G$. Thus  $s\geq t$.\\
{\bf Claim 2:} $s+t\geq 2n-k+1$.

By contradiction. Suppose that $s+t\leq 2n-k$. Since $s\geq t$,  we have $t\leq n-\frac{k}{2}$ and $t\leq s\leq 2n-k-t$. By Claim 1, $n-k+1\leq t\leq n-\frac{k}{2}$.
 Without loss of generality, let $X_3\subseteq X$ and $Y_3\subseteq Y$ such that $G(X_3,Y_3,E_3)=K_{s,t}$. Then  for any $x\in X\backslash X_3$,   and  any $y\in Y\backslash Y_3$, $d(x)\leq n-s+1$ and $d(y)\leq n-t+1$ hold, respectively. Thus
\begin{align*}
e(G)&=e(X_3,Y_3)+e(X\backslash X_3,Y)+e(X,Y\backslash Y_3)\\
    &\leq st+(n-s)(n-s+1)+(n-t)(n-t+1)\\
    &=s^2-(2n+1-t)s+n^2+n+(n-t)(n-t+1)\\
    &\leq (2n-k-t)^2-(2n+1-t)(2n-k-t)+n^2+n+(n-t+1)(n-t)\\
    &=t^2+(k-2n)t+2n(n+1)-(2n-k)(k+1)\\
    &\leq (n-k+1)^2+(k-2n)(n-k+1)+2n(n+1)-(2n-k)(k+1)\\
    &=n^2-nk+k^2+1\\
    &<n(n-k)+k(k+1)\\
    &<e(G),
\end{align*}
a contradiction.

Thus $s+t\geq 2n-k+1$.

{\bf Claim 3:}  $K_{n,n-k+1}\subseteq G$.

By contradiction. Assume that $K_{n,n-k+1}\nsubseteq G$, then $s\leq n-1$.  By Claim 2,  $s+t\geq 2n-k+1$ holds. Then   $t\geq 2n-k+1-s\geq 2n-k+1-(n-1)=n-k+2$.
 Thus $n-k+2\leq t\leq s\leq n-1$. Note that  $\delta(G)\geq k$ and $d(w)\geq s \geq t$ for any vertex $w\in Y_3$. Then $d(x)+d(w)\geq k+(n-k+2)=n+2$ for any $x\in X\backslash X_3$. By the definition of $cl_{B_{n+2}}(G)$,   for any $x\in X\backslash X_3$,
 $x$ is adjacent to every vertex of $Y_3$.
 This implies that $s=n$, which contradict with $s\leq n-1$.

Hence if $\delta(G)\geq k$, then $K_{n,n-k+1}\subseteq G$.

 The proof is finished.
\end {proof}

Finally, we obtain Theorem \ref{le4} by Lemma \ref{le3}.

\begin{theorem}\label{le4}
Let $G(X, Y, E)$ be a balanced bipartite graph of order $2n$. If $\delta(G)\geq k \geq 1$, $n\geq 2k$, and $e(G)>n(n-k)+k(k+1)$, then $G$ is weakly Hamilton--connected.
\end{theorem}

\begin{proof}
[\bf Proof] Let $G'=cl_{B_{n+2}}(G)$, by Lemma \ref{le3}, $K_{n,n-k+1}\subseteq G'$.
Let $Y_1\subseteq Y$ such that $G(X,Y_1,E_1)=K_{n,n-k+1}$.\\
{\bf Claim:} All vertices in  $ Y\backslash Y_1$ have the same neighbour set.\\
By contradiction. If there are  two vertices  $y, y' \in Y\backslash Y_1$ such that
 $N_{G'}(y)\neq N_{G'}(y')$, then there exist $x\in N_{G'}(y)\backslash N_{G'}(y')$ or
  $x'\in N_{G'}(y')\backslash N_{G'}(y)$. Without loss of generality,
 assume that there exists $x\in N_{G'}(y)\backslash N_{G'}(y')$,
 then $d_{G'}(x)\geq n-k+2$. Since $d_{G'}(y')\geq k$ and $G'=cl_{B_{n+2}}(G)$,
 $x$ is adjacent to $y'$,  a contradiction.

By the above claim, $Q_n^k\subseteq G'$. Since $Q_n^k$  is weakly Hamilton--connected,  $G'$ is weakly Hamilton--connected.
By Theorem \ref{th3}, $G$ is weakly Hamilton--connected.
\end{proof}

\section{Spectral radius and weakly Hamilton--connected bipartite graphs}
In Section 4, we give some sufficient conditions for weakly Hamilton-connectedness of bipartite graphs in terms of spectral radius and
signless Laplacian spectral radius of $G$ and $\widehat G$.

Denote the minimum degree sum of nonadjacent vertices between different parts of $G$ by $\sigma(G)$.
Let $G_1=G(X_1,Y_1,E_1)=K_{t,t}$,  where $X_1=\{x_1,x_2,\ldots, x_t\},  Y_1=\{y_1,y_2,\ldots, y_t\}$;
$G_2=G(X_2,Y_2,E_2)=K_{n-t+1,n-t+1}$,  where $X_2=\{x_m,x_{t+1},\ldots, x_n\},  Y_2=\{y_m,y_{t+1},\ldots, y_n\}$.
Let $R^t_n$ (Figure 1) be the graph obtained by gluing $x_t \text{ with } x_m$ and $y_t \text{ with } y_m$, respectively.
Let $S_n^t$ be the graph obtained from $R^t_n$ by deleting the edge $x_ty_t$.
%(2\leq t\leq \frac{n+1}{2}) or 2\leq t\leq n-1

\vspace{0.5cm}
\begin{tikzpicture}

\draw (6,1) ellipse (3.1 and 2);
\draw (11,1) ellipse (3.1 and 2);

\draw[fill] (5,1.8) circle[radius=0.05];
\draw[fill] (7,1.8) circle[radius=0.05];

\draw[fill] (5.9,1.8) circle[radius=0.02];
\draw[fill] (6,1.8) circle[radius=0.02];
\draw[fill] (6.1,1.8) circle[radius=0.02];

\draw[fill] (8.5,1.8) circle[radius=0.05];

\draw[fill] (10,1.8) circle[radius=0.05];
\draw[fill] (12,1.8) circle[radius=0.05];

\draw[fill] (10.9,1.8) circle[radius=0.02];
\draw[fill] (11,1.8) circle[radius=0.02];
\draw[fill] (11.1,1.8) circle[radius=0.02];

\draw[fill] (5,0.2) circle[radius=0.05];
\draw[fill] (7,0.2) circle[radius=0.05];

\draw[fill] (5.9,0.2) circle[radius=0.02];
\draw[fill] (6,0.2) circle[radius=0.02];
\draw[fill] (6.1,0.2) circle[radius=0.02];

\draw[fill] (8.5,0.2) circle[radius=0.05];

\draw[fill] (10,0.2) circle[radius=0.05];
\draw[fill] (12,0.2) circle[radius=0.05];

\draw[fill] (10.9,0.2) circle[radius=0.02];
\draw[fill] (11,0.2) circle[radius=0.02];
\draw[fill] (11.1,0.2) circle[radius=0.02];

\draw (5,1.8) --(5,0.2) ;
\draw (5,1.8) --(7,0.2) ;
\draw (5,1.8) --(8.5,0.2) ;

\draw (7,1.8) --(5,0.2) ;
\draw (7,1.8) --(7,0.2) ;
\draw (7,1.8) --(8.5,0.2) ;

\draw (5,0.2)-- (8.5,1.8);
\draw (7,0.2)-- (8.5,1.8);

\draw (10,1.8) --(10,0.2) ;
\draw (10,1.8) --(12,0.2) ;
\draw (10,0.2) --(8.5,1.8) ;

\draw (12,1.8) --(10,0.2) ;
\draw (12,1.8) --(12,0.2) ;
\draw (12,0.2) --(8.5,1.8) ;

\draw (8.5,0.2)-- (10,1.8);
\draw (8.5,0.2)-- (12,1.8);

\draw (8.5,0.2) --(8.5,1.8);

\coordinate[label=0.5:$x_1$] (v1) at (4.8,2);
\coordinate[label=1.5:$x_{t-1}$] (v2) at (6.4,2);
\coordinate[label=0.5:$y_1$] (v1) at (4.8,0);
\coordinate[label=1.5:$y_{t-1}$] (v2) at (6.4,0);

\coordinate[label=1.5:$x_{t}$] (v2) at (8.3,2);
\coordinate[label=1.5:$y_{t}$] (v2) at (8.3,0);

\coordinate[label=1:$x_{t+1}$] (v{n-1})  at (9.6,2);
\coordinate[label=3.25:$x_n$] (vn) at (11.7,2);
\coordinate[label=1:$y_{t+1}$] (v{n-1})  at (9.6,0);
\coordinate[label=3.25:$y_n$] (vn) at (11.7,0);

\coordinate[label=1.8:$\text{Figure 1. The graph }R_n^t$] (G2) at (6,-2);
\coordinate[label=0.25:$K_{t,t}$] (K_{n-k-1}) at (5.25,-0.5);
\coordinate[label=2.65:$K_{n-t+1,n-t+1}$] (K_{n-t,n-t}) at (10.2,-0.5);

\coordinate[label=2.5: ] (vn) at (1.8,-2.75);

\end{tikzpicture}

The following lemma is inspired by Ferrara, Jacobson and Powell\cite{lesM2012}.
\begin{lemma}\label{th2}
Let $G=G(X,Y,E)$ be a non--weakly Hamilton--connected bipartite graph with $X=\{x_1,x_2, \dots, x_n\}$ and $Y=\{y_1,y_2, \dots, y_n\}$. If $\sigma(G)=n+1$, then  $S^t_n \subseteq G\subseteq R^t_n$ where $2\leq t\leq n-1$.
\end{lemma}

\begin{proof}
[\bf Proof]Adding some suitable edges to $G$ to get a maximal non--weakly Hamilton--connected bipartite graph $G^*$.
By Corollary \ref{co2}, $\sigma(G^*)\leq n+1$. Since $\sigma(G)=n+1$, $\sigma(G^*)=n+1$ holds.
By Corollary \ref{co1}, there is an edge $e_0$ which  is not contained in any Hamilton cycle of $G^*$.
 Note that $G^*$ is not weakly Hamilton--connected, then it is  not a complete bipartite graph.
  For two nonadjacent vertices $x\in X$ and $ y\in Y$ with $d_{G^*}(x)+d_{G^*}(y)=n+1$,
   $G^*+xy$ is weakly Hamilton--connected.
Thus there is a Hamilton path $P$ with the endpoints $x,y$ in $G^*$,
where $P$ contains $e_0$ with the form $xy_1x_2y_2\cdots x_ny$ $(x=x_1,y=y_n)$.

In $P$, let
\begin{equation*}
S_i=\{x_i,y_i\} \text{ be}
\begin{cases} \text{ an x--pair}, &\text{ if } xy_i\in E(G^*);\\
\text{ a y--pair}, &\text{ if } x_iy\in E(G^*).
\end{cases}
\end{equation*}

% Let $S_i=\{x_i,y_i\}$ be an $x$--pair with respect to $P$ if $xy_i\in E(G^*)$, and a $y$--pair with respect to $P$ if $x_iy\in E(G^*)$.
For any other Hamilton path $P'$ with the endpoints $x',y'$ in $G^*$, we similarly define $x'$--pair and $y'$--pair with respect to $P'$.

Let $I=\{i |1\leq i \leq n-1,xy_i\in E(G^*) \quad and  \quad x_iy_i\neq e_0\}$, then for arbitrary $i\in I$, $x_iy\notin E(G^*)$. Otherwise, there exists a Hamilton cycle $y_ix\overrightarrow{P}x_iy\overleftarrow{P}y_i$ containing $e_0$.
Since $d_{G^*}(x)+d_{G^*}(y)=n+1$, there exists a unique integer $m$ ($2\leq m\leq n-1$) such that $x_m$ is adjacent to $y$ and $y_m$ is adjacent to $x$ in $P$. Then $e_0=x_my_m$. And for any $i$  ($1\leq i\leq n$ and $i\neq m$), either $x$ is adjacent to $y_i$ or $y$ is adjacent to $x_i$, but not both. \\
{\bf Claim:} $S_1, \dots, S_m$ are $x$--pairs and $S_m, \dots, S_n$ are $y$--pairs with respect to $P$.

\indent (I)   $S_1, \dots, S_m$ are $x$--pairs with respect to $P$.

For $m=2$, it is obvious.

For $3\leq m\leq n-1$, we prove Case (I)  by contradiction.

Suppose that there exists $1\leq j\leq m-2$ such that $S_1, \dots, S_j$ are $x$--pairs and $S_{j+1}$ is $y$--pair with respect to $P$.
Since there is a Hamilton path $P_1=y_{m-1}\overleftarrow{P}xy_mx_my\overleftarrow{P}x_{m+1}$ containing $e_0=x_my_m$ in $G^*$,
$x_{m+1}y_{m-1}\notin E(G^*)$ holds. It implies that $d_{G^*}(y_{m-1})+d_{G^*}(x_{m+1})\geq n+1$.
Since $S_m=\{y_m,x_m\}$ is both an $x_{m+1}$--pair and a $y_{m-1}$--pair with respect to $P_1$,
 $S_j=\{y_j,x_j\}$ is either an $x_{m+1}$--pair or a $y_{m-1}$--pair with respect to $P_1$, but not both. \\
(1)   If $x_{m+1}y_j\in E(G^*)$, then there is a Hamilton cycle $y_j\overleftarrow{P}xy_mx_m\overleftarrow{P}x_{j+1}y\overleftarrow{P}x_{m+1}y_j$ containing $e_0=x_my_m$ which is contrary to the hypothesis. \\
(2)  If $x_jy_{m-1}\in E(G^*)$, then $x_{m+1}\text { and } y_j$ is a pair of nonadjacent vertices in $G^*$ which implies $d_{G^*}(y_j)+d_{G^*}(x_{m+1})\geq n+1$.
Since there is a Hamilton path $$P_2=y_j\overleftarrow{P}xy_mx_m\overleftarrow{P}x_{j+1}y\overleftarrow{P}x_{m+1},$$
there exists an integer $1\leq k\leq n$ such that $S_k=\{y_k,x_k\}$ is both $y_j$--pair and $x_{m+1}$--pair with respect to $P_2$.
By the above proof, we have $k=m$. That is, $x_my_j\in G^*$.
Then there is a Hamilton cycle $x_my_m\overrightarrow{P}yx_{j+1}\overrightarrow{P}y_{m-1}x_j\overleftarrow{P}xy_jx_m$, a contradiction.

Thus $S_1, \dots, S_m$ are all $x$--pairs with respect to $P$.

\indent (II) $S_m, \dots, S_n$ are $y$--pairs with respect to $P$.

By contradiction. Suppose that there exists $j$ ($m\leq j<n-1$) such that $S_m, \dots, S_j$ are $y$--pairs and $S_{j+1}$ is $x$--pair.
Since there is a Hamilton path $P_3=x_{j+1}\overleftarrow{P}xy_{j+1}\overrightarrow{P}y$ which contains $e_0$, $y$ is not adjacent to $x_{j+1}$.
It implies that $d_{G^*}(y)+d_{G^*}(x_{j+1})\geq n+1$.
Then

(1)  If there exists $1\leq k\leq j-1$ such that $\{x_{k+1},y_k\}$ is both $x_{j+1}$--pair and $y$--pair with respect to $P_3$,
then there is a Hamilton cycle $y_k\overleftarrow{P}xy_{j+1}\overrightarrow{P}yx_{k+1}\overrightarrow{P}x_{j+1}y_{k}$,  which contains $e_0$, a contradiction.

(2)   If there exists $j+2\leq k\leq n-1$ such that $\{x_{k},y_k\}$ is both $x_{j+1}$--pair and $y$--pair with respect to $P_3$,
then there is a Hamilton cycle $x_{j+1}\overleftarrow{P}xy_{j+1}\overrightarrow{P}x_{k}y\overleftarrow{P}y_{k}x_{j+1}$ which contains $e_0$, a contradiction.

Thus $S_m, \dots, S_n$ are all $y$--pairs with respect to $P$.

Hence, $S_1, \dots, S_m$ are $x$--pairs and $S_m, \dots, S_n$ are $y$--pairs with respect to $P$.

For  arbitrary $i$, $ j$ with $1\leq i<m<j\leq n$, $x_i$ and $y_j$ is nonadjacent.
Otherwise, there exists a Hamilton cycle $x\overrightarrow{P}x_iy_j\overrightarrow{P}yx_j\overleftarrow{P}y_ix$,  which contains $e_0$ in $G^*$.
Then \{$x_1,\dots,x_{m-1},y_{m+1}, \ldots,y_n$\} is an independent set.

For arbitrary $2\leq i<m$, there exists a Hamilton path $P'=x_i\overleftarrow{P}xy_i\overrightarrow{P}y$ contains $e_0$ in $G^*$. By the above claim,
for any $1\leq l\leq i-1$ and $i+1\leq k\leq m$, $\{x_l, y_{l-1}\} $ and $\{x_k, y_{k}\}$ are  $x_i$--pairs with respect to $P'$.
For arbitrary $m<j\leq n-1$, there exists a Hamilton path $P''=y_j\overrightarrow{P}y x_j\overleftarrow{P}x$ contains $e_0$ in $G^*$. By the above claim,
for any $m\leq k\leq j-1$ and $j+1\leq l\leq n-1$, $\{x_k, y_{k}\}$ and $\{y_l, x_{l+1}\} $ are  $y_j$--pairs with respect to $P''$.
Thus there exists a Hamilton path $P'''=x_ny_nx_{n-1}y_{n-1}\cdots x_2y_2x_1y_1$ contains  $e_0$ in $G^*$.

For  arbitrary $i$, $ j$ with $1\leq i<m<j\leq n$, $x_j$ and $y_i$ must be nonadjacent.
Otherwise, $x\overrightarrow{P}y_ix_j\overrightarrow{P}yx_{j-1}\overrightarrow{P'''}y_{i+1}x$ is a Hamilton cycle which contains $e_0$ in $G^*$.
Then \{$y_1,\dots,y_{m-1},x_{m+1},\dots,x_n$\} is an independent set.
Since $\sigma(G^*)=n+1$, $G^*\subseteq R_n^m$ holds.
And by the assumption on the maximality of $G^*$, we have $G^*\cong R_n^m$. Removing any edge $x_iy_j$ from $G^*$, where either $1\leq i\leq m$, $1\leq j\leq m$ or $m\leq j\leq n $, $m\leq i\leq n $ unless $i=j=m$,  we obtain a graph with minimum degree sum of nonadjacent vertices less than $n+1$. Thus $S^m_n\subseteq G\subseteq R^m_n $.

Hence,  $S^m_n\subseteq G\subseteq  R_n^m$ holds for $2\leq m\leq n-1$.

 The proof is completed.
\end{proof}

By direct calculations, we have the following lemma.

\begin{lemma}\label{le5}
(1) $\rho(Q_n^t)>\rho(K_{n,n-t+1})=\sqrt{n(n-t+1)}$.\\
(2) $q(Q_n^t)>q(K_{n,n-t+1})=2n-t+1$.\\
(3 )$\rho(\widehat{R^t_n})=\rho(\widehat{S^t_n})=\rho(\widehat{Q^t_n})=\rho(K_{t-1,n-t})=\sqrt{(t-1)(n-t)}$.\\
(4) $q(\widehat{R^t_n})=q(\widehat{S^t_n})=q(\widehat{Q^t_n})=q(K_{t-1,n-t})=n-1$.
\end{lemma}

\begin{lemma}(Bhattacharya, Friedland, and Peled \cite{lesB2008})\label{f1}
Let $G=G(X,Y,E)$ be a bipartite graph. Then
\[
\rho(G)\leq \sqrt {e(G)}.
\]
\end{lemma}

\begin{lemma}(Li and Ning\cite{lesN.L2016})\label{f2}
Let $G=G(X,Y,E)$ be a balanced bipartite graph of order $2n$. Then
\[
q(G)\leq \frac {e(G)}{n}+n.
\]
\end{lemma}

\begin{lemma}(Li and Ning\cite{lesN.L2016})\label{f3}
Let $G$ be a graph with nonempty edge set. Then
\[
\rho(G)\geq min\{\sqrt{d(u)d(v)}: uv\in E(G)\}.
\]
If $G$ is connected, then the equality holds if and only if $G$ is regular or semi--regular bipartite.
\end{lemma}

\begin{lemma}(Li and Ning\cite{lesN.L2016})\label{f4}
Let $G$ be a graph with nonempty edge set. Then
\[
q(G)\geq min\{d(u)+d(v): uv\in E(G)\}.
\]
If $G$ is connected, then the equality holds if and only if $G$ is regular or semi--regular bipartite.
\end{lemma}

\begin{theorem}\label{th6}
Let $G=G(X,Y,E)$ be a bipartite graph with $|X|=|Y|=n$ and the minimum degree $\delta(G) \geq k\geq 2$. \\
(1) If $n\geq k(k+1)$ and $\rho(G)\geq \rho(Q^k_n)$, then $G$ is weakly Hamilton--connected.\\
(2) If $n\geq k(k+1)$ and $q(G)\geq q(Q^k_n)$, then $G$ is weakly Hamilton--connected.\\
(3) If $n\geq 2k-1$ and $\rho(\widehat G)\leq \rho(\widehat {Q^k_n})$, then $G$ is weakly Hamilton--connected unless $S^k_n\subseteq G\subseteq R^k_n$.\\
(4) If $n\geq 2k-1$ and $q(\widehat G)\leq q(\widehat {Q^k_n})$, then $G$ is weakly  Hamilton--connected unless $S^k_n\subseteq G\subseteq R^k_n$.
\end{theorem}

\begin{proof}
%[\bf Proof of Theorem \ref{th6}]
(1) By Lemmas  \ref{le5} and \ref{f1},
\[
\sqrt{n(n-k+1)}<\rho(Q^k_n)\leq\rho(G)\leq\sqrt {e(G)}.
\]
Then $e(G)>n(n-k+1)\geq n(n-k)+k(k+1)$ with $n\geq k(k+1)$. By Theorem \ref{le4} and $Q^k_n$ is weakly Hamilton--connected, then $G$ is weakly Hamilton--connected.

(2) By Lemmas \ref{le5} and \ref{f2}
\[
2n-k+1<q(G)\leq \frac{e(G)}{n}+n.
\]
Then
\[
e(G)>n(n-k+1)\geq n(n-k)+k(k+1)
\]
with $n\geq k(k+1)$. By Theorem \ref{le4} and $Q^k_n$ is weakly Hamilton--connected,  then $G$ is weakly Hamilton--connected.

(3) By contradiction. Suppose that $G$ is not weakly Hamilton--connected and $S^k_n\nsubseteq G\nsubseteq R^k_n$.  Let $G'=cl_{B_{n+2}}(G)$.  By Theorem \ref{th3}, $G'$ is not weakly Hamilton--connected. For any a pair of nonadjacent vertices $x\in X$ and $y\in Y$, by $G'=cl_{B_{n+2}}(G)$, then $d_{G'}(x)+d_{G'}(y)\leq n+1$. Without loss of generality, assume that $d_{ {G'}}(x)\leq d_{ {G'}}(y)$. Then
\begin{equation}\label{equ08}
d_{\widehat {G'}}(x)+d_{\widehat {G'}}(y)=2n-\big(d_ {G'}(x)+d_{G'}(y)\big)\geq n-1.
\end{equation}
Since $\delta(G')\geq\delta(G)\geq k$, $d_{\widehat {G'}}(x)\leq n-k$ and  $d_{\widehat {G'}}(y)\leq n-k$ hold. Thus
\begin{align}\label{equ09}
\begin{split}
d_{\widehat {G'}}(x)\geq n-1-d_{\widehat {G'}}(y)\geq n-1-n+k=k-1,\\
d_{\widehat {G'}}(y)\geq n-1-d_{\widehat {G'}}(x)\geq n-1-n+k=k-1.
\end{split}
\end{align}
Then
\begin{align}\label{equ00}
\begin{split}
k-1\leq d_{\widehat {G'}}(x)\leq n-k,\\
k-1\leq d_{\widehat {G'}}(y)\leq n-k.
\end{split}
\end{align}
Note that $f(x)=x(n-1-x)\geq (k-1)(n-k)$ for $k-1\leq x \leq n-k$. By (\ref{equ08}) and (\ref{equ09}),
\begin{equation}\label{equ0bu}
d_{\widehat {G'}}(x)d_{\widehat {G'}}(y)\geq d_{\widehat {G'}}(x)(n-1-d_{\widehat {G'}}(x))\geq (k-1)(n-k),
\end{equation}
the equality holds if and only if $d_{\widehat {G'}}(x)=n-k$, $d_{\widehat {G'}}(y)=k-1$.\\
By Lemma \ref{f3},
\[
\rho(\widehat {G'})\geq min\left\{\sqrt{d_{\widehat {G'}}(x)d_{\widehat {G'}}(y)}\bigg| xy\in E(\widehat {G'})\right\}\geq \sqrt{(k-1)(n-k)}.
\]
Then
\[
\sqrt{(k-1)(n-k)}\geq \rho(\widehat {G})\geq \rho(\widehat {G'})\geq \sqrt{(k-1)(n-k)}.
\]
Hence
\[
\rho(\widehat {G'})=\sqrt{(k-1)(n-k)}.
\]
And there is an edge $xy\in E(\widehat {G'})$, where $d_{\widehat {G'}}(x)=n-k$ and $d_{\widehat {G'}}(y)=k-1$.\\
Let $F$ be the complement of $\widehat {G'}$ containing $xy$.  By Lemma \ref {f3}, $F$ is a semi--regular bipartite graph with parts $X'\subseteq X$,  and $Y'\subseteq Y$.

There are two cases:

Case (3.1):   $F=K_{n-k,k-1}$.

 Then $|V(F)|= n-1$. If $F$ is the unique nontrivial complement of ${\widehat {G'}}$, then $G'=Q_n^k$, which implies that $G'$ is weakly Hamilton--connected, a contradiction.
 Then ${\widehat {G'}}$ has at least two nontrivial components.  By (\ref{equ08}), every nontrivial complement of ${\widehat {G'}}$ has order at least $n-1$. Thus ${\widehat {G'}}$ has only two nontrivial components $F, F'$. For $|V(F')|$, there are three subcases:\\
$(a)\quad |V(F')|=n-1; \quad (b) \quad |V(F')|=n; \quad (c)\quad |V(F')|=n+1.$ \\
Subcase (a):   By (\ref{equ00}) and Lemma \ref{f3}, it is  easy to get $F'=K_{n-k,k-1}$.\\
Subcases (b) and (c):  If there exists $uv\in E(F')$  such that $d_{F'}(u)+d_{F'}(v)\geq n$, then
\begin{equation}\label{equ0bu1}
d_{{F'}}(u)d_{{F'}}(v)\geq d_{{F'}}(u)(n-d_{{F'}}(u))\geq k(n-k).
\end{equation}
By (\ref{equ0bu1}) and Lemma \ref{f3},
\[
\rho(F')\geq min\left\{\sqrt {d_{F'}(u)d_{F'}(v)}\bigg|uv\in E(F')\right\}\geq\sqrt{k(n-k)}>\sqrt{(k-1)(n-k)},
\]
a contradiction.

Thus for any $xy\in E(F')$, $d_{F'}(x)+d_{F'}(y)=n-1$. That is, for any two nonadjacent vertices $x,y$ in $G'$, $d_{G'}(x)+d_{G'}(y)\geq n+1$.

Case (3.2):   $F\neq K_{n-k,k-1}$.

 Then $|V(F)|> n-1$. Note that every nontrivial complement of ${\widehat {G'}}$ has order at least $n-1$, then ${\widehat {G'}}$ has at most two nontrivial components. \\
1)\quad If ${\widehat {G'}}$ has only one nontrivial component $F$, then  any two nonadjacent vertices in distinct parts of $G'$ has degree sum at least $n+1$. \\
2)\quad If ${\widehat {G'}}$ has only two nontrivial components $F, F'$, then there are three subcases:
\begin{equation*}
(a)
\begin{cases} |V(F)|=n;\\
|V(F')|=n-1.
\end{cases}
(b)
\begin{cases} |V(F)|=n+1;\\
|V(F')|=n-1.
\end{cases}
(c)
\begin{cases} |V(F)|=n;\\
|V(F')|=n.
\end{cases}
\end{equation*}
Subcases (a) and (b):    By (\ref{equ00}) and Lemma \ref{f3}, $F'=K_{n-k,k-1}$.\\
Subcase (c):     If there exists $xy\in E(F')$, such that $d_{F'}(x)+d_{F'}(y)=n$. By (\ref{equ0bu1}),
\[
\rho(F')\geq min\left\{\sqrt {d_{F'}(x)d_{F'}(y)}\bigg|xy\in E(F')\right\}\geq\sqrt{k(n-k)}.
\]
Thus
\[
\rho(\widehat {G'})=max\{\rho(F), \rho(F')\}>\sqrt{(k-1)(n-k)},
\]
a contradiction.

Thus for any $xy\in E(F'), d_{F'}(x)+d_{F'}(y)=n-1$. That is, for any two nonadjacent vertices $x,y$ in $G'$, $d_{G'}(x)+d_{G'}(y)\geq n+1$.

Hence  $\sigma(G')=n+1$ holds  whether $F=K_{n-k,k-1}$ or $F\neq K_{n-k,k-1}$.  By Lemma \ref{th2}, $S_n^k \subseteq G'\subseteq R_n^k$.

If $G\subsetneqq S_n^k$, then $\rho(\widehat G)>\rho(\widehat {S^k_n})$.
By Lemma \ref{le5}, $\rho(\widehat G)>\rho(\widehat {Q^k_n})$, contradict with $\rho(\widehat G)\leq\rho(\widehat {Q^k_n})$.
Thus $G\supseteq S_n^k$.

Then $S_n^k \subseteq G\subseteq R_n^k$, a contradiction.

(4)\quad By contradiction. Suppose that $G$ is not weakly Hamilton--connected and $S^k_n\nsubseteq G\nsubseteq R^k_n$.  Let $G'=cl_{B_{n+2}}(G)$, by Theorem \ref{th3}, $G'$ is not weakly Hamilton--connected. For any  pair of nonadjacent vertices $x\in X$ and  $y\in Y$, $d_{G'}(x)+d_{G'}(y)\leq n+1$. Without loss of generality, $d_{ {G'}}(x)\leq d_{ {G'}}(y)$. Then
\begin{equation}\label{equ8}
d_{\widehat {G'}}(x)+d_{\widehat {G'}}(y)=2n-(d_{\widehat {G'}}(x)+d_{\widehat {G'}}(y))\geq n-1
\end{equation}
Since $\delta(G')\geq\delta(G)\geq k$, $d_{\widehat {G'}}(x)\leq n-k$ and $d_{\widehat {G'}}(y)\leq n-k$ hold. Thus
\begin{align}\label{equ9}
\begin{split}
d_{\widehat {G'}}(x)\geq n-1-d_{\widehat {G'}}(y)\geq n-1-n+k=k-1,\\
 d_{\widehat {G'}}(y)\geq n-1-d_{\widehat {G'}}(x)\geq n-1-n+k=k-1.
\end{split}
\end{align}
Then
\begin{align}\label{equ0}
\begin{split}
k-1\leq d_{\widehat {G'}}(x)\leq n-k,\\
k-1\leq d_{\widehat {G'}}(y)\leq n-k.
\end{split}
\end{align}
By Lemma \ref{f4} and (\ref{equ8}),
\[
q(\widehat {G'}) \geq min\left\{ d_{\widehat{G'}}(x)+d_{\widehat {G'}}(y)\bigg| xy\in E(\widehat {G'}) \right\}\geq n-1.
\]
By Lemma \ref{le5},
\[
n-1 \geq q(\widehat{G})\geq q(\widehat {G'})\geq n-1.
\]
Then
\[
q(\widehat {G'})=n-1.
\]
And there is an edge $xy\in E(\widehat {G'})$, where $d_{\widehat {G'}}(x)=n-k$ and $d_{\widehat {G'}}(y)=k-1$.\\
Let $F$ be the complement of $\widehat {G'}$ containing $xy$. By Lemma \ref {f4}, $F$ is a semi--regular bipartite graph with parts $X'\subseteq X$, and  $Y'\subseteq Y$.

There are two cases:

Case (4.1): $F=K_{n-k,k-1}$.

 Then $|V(F)|= n-1$. If $F$ is the unique nontrivial complement of ${\widehat {G'}}$, then $G'=Q_n^k$, which implies that $G'$ is weakly Hamilton--connected, a contradiction. By (\ref{equ8}), every nontrivial complement of ${\widehat {G'}}$ has order at least $n-1$. Thus ${\widehat {G'}}$ has only two nontrivial components $F$ and $ F'$. For $|V(F')|$, there are three subcases:\\
$(a)\quad |V(F')|=n-1; \quad (b) \quad |V(F')|=n; \quad (c)\quad |V(F')|=n+1.$ \\
Subcase (a):  By (\ref{equ0}) and Lemma \ref{f3}, $F'=K_{n-k,k-1}$.\\
Subcases (b) and (c): If there exists $uv\in E(F')$, such that $d_{F'}(u)+d_{F'}(v)\geq n$, then
\begin{equation}\label{equ0bu2}
d_{{F'}}(u)+d_{{F'}}(v)\geq d_{{F'}}(u)+(n-d_{{F'}}(u))\geq n.
\end{equation}
By (\ref{equ0bu2}) and Theorem \ref{f3},
\[
q(F')\geq min\left\{d_{F'}(x)+d_{F'}(y)\bigg|xy\in E(F')\right\}\geq n> n-1,
\]
a contradiction.

Thus for any $xy\in E(F'), d_{F'}(x)+d_{F'}(y)=n-1$ holds. That is, for any two nonadjacent vertices $x,y$ in $G'$,  we have $d_{G'}(x)+d_{G'}(y)\geq n+1$.

Case (4.2): $F\neq K_{n-k,k-1}$.

Then $|V(F)|> n-1$. Note that every nontrivial complement of ${\widehat {G'}}$ has order at least $n-1$. Then ${\widehat {G'}}$ has at most two nontrivial components. \\
1)\quad If ${\widehat {G'}}$ has only one nontrivial component $F$, then for any two nonadjacent vertices in distinct parts of $G'$ has degree sum at least $n+1$. \\
2)\quad If ${\widehat {G'}}$ has only two nontrivial components $F$ and $F'$, then there are three subcases:
\begin{equation*}
(a)
\begin{cases} |V(F)|=n;\\
|V(F')|=n-1.
\end{cases}
(b)
\begin{cases} |V(F)|=n+1;\\
|V(F')|=n-1.
\end{cases}
(c)
\begin{cases} |V(F)|=n;\\
|V(F')|=n.
\end{cases}
\end{equation*}
Subcases (a) and (b): By (\ref{equ0}) and Theorem \ref{f3}, $F'=K_{n-k,k-1}$.\\
Subcase (c): If there exist $uv\in E(F')$ such that $d_{F'}(u)+d_{F'}(v)=n$,  by (\ref{equ0bu2}) and Theorem \ref{f3}, then
\[
q(F')\geq min\left\{d_{F'}(x)+d_{F'}(y)\bigg|xy\in E(F')\right\}\geq n.
\]
Thus
\[
q(\widehat {G'})=max\{q(F), q(F')\}>n-1,
\]
a contradiction.

Thus for any $xy\in E(F')$, we have $d_{F'}(x)+d_{F'}(y)=n-1$. That is, for any two nonadjacent vertices $x,y$ in $G'$, $d_{G'}(x)+d_{G'}(y)\geq n+1$ holds.

Hence $\sigma(G')=n+1$ holds whether $F=K_{n-k,k-1}$ or $F\neq K_{n-k,k-1}$.   By Lemma \ref{th2}, $S_n^k \subseteq G'\subseteq R_n^k$.

If $G\subsetneqq S_n^k$, then $q(\widehat G)>q(\widehat {S^k_n})$. By Lemma \ref{le5}, $q(\widehat G)>q(\widehat {Q^k_n})$,
 contradicts with $q(\widehat G)\leqq(\widehat {Q^k_n})$. Thus $G\supseteq S_n^k$.

Then $S_n^k \subseteq G\subseteq R_n^k$, a contradiction.

The proof is finished.
\end{proof}

By the proof of Theorem \ref{th6} (1) and (2), we get the following result.
\begin{corollary}\label{co00}
Let $G=G(X,Y,E)$ be a bipartite graph with $|X|=|Y|=n$ and the minimum degree $\delta(G) \geq k\geq 2$. \\
(1) If $n> k(k+1)$ and $\rho(G)\geq \sqrt{n(n-k+1)}$, then $G$ is weakly Hamilton--connected.\\
(2) If $n> k(k+1)$ and $q(G)\geq 2n-k+1$, then $G$ is weakly Hamilton--connected.
\end{corollary}

%\begin{proof}
%By Lemma \ref{f1},
%\[
%\sqrt{n(n-k+1)}\leq\rho(G)\leq \sqrt{e(G)}.
%\]
%Note that $n>k(k+1)$. Then
%\[
%e(G)\geq n(n-k+1)>n(n-k)+k(k+1)
%\]
%By Lemma \ref{le4}, $G$ is weakly Hamilton--connected.
%\end{proof}

\end{document}